\documentclass[12pt]{amsart}%
\usepackage{amsfonts}
\usepackage{txfonts}
\usepackage{mathtools}
\usepackage{amsmath}
\usepackage{amssymb}
\usepackage{amsthm}
\usepackage{newlfont}
\usepackage{graphicx}%
\providecommand{\U}[1]{\protect\rule{.1in}{.1in}}
\newtheorem{example}{Example}
\newtheorem{remark}{Remark}
\newtheorem{thm}{Theorem}[section]
\newtheorem{theorem}{Theorem}[section]

\newtheorem{cor}[thm]{Corollary}

\newtheorem{prop}[thm]{Proposition}

\theoremstyle{definition}

\numberwithin{equation}{section}

\newcommand{\ed}{\end {document}}

\begin{document}
\title{\bf Gradient Estimates And Liouville Theorems For A Class of Nonlinear Elliptic Equations}

\author{Pingliang Huang}
\address{Department of Mathematics, Shanghai University,
Shanghai, 200444, China, email: huangpingliang@shu.edu.cn}
\author{Youde  Wang$^{1}$}\footnote{The corresponding author.}
\address{1. School of Mathematics and Information Sciences, Guangzhou University; 2. Hua Loo-Keng Key Laboratory of Mathematics, Institute of Mathematics, Academy of Mathematics and Systems Science, Chinese Academy of Sciences, Beijing 100190, China; 3. School of Mathematical Sciences, University of Chinese Academy of Sciences, Beijing 100049, China.}
\email{wyd@math.ac.cn}

\begin{abstract} In this paper, first we study carefully the positive solutions to $\Delta u+\lambda_{1}u\ln u +\lambda_{2}u^{b+1}=0$ defined on a complete noncompact Riemannian manifold $(M, g)$ with $Ric(g)\geq -Kg$, which can be regarded as Lichnerowicz-type equations, and obtain the gradient estimates of positive solutions to these equations which do not depend on the bounds of the solutions and the Laplacian of the distance function on $(M, g)$. Then, we extend our techniques to a class of more general semilinear elliptic equations $\Delta u(x)+uh(\ln u)=0$ and obtain some similar results under some suitable analysis conditions on these equations. Moreover, we also obtain some Liouville-type theorems for these equations when $Ric(g)\geq 0$ and establish some Harnack inequalities as consequences.

\end{abstract}
\maketitle
\vspace{3mm} {\bf Keywords } {Gradient estimate, Ricci curvature,  Liouville theorem, Harnack inequality, Nonlinear elliptic equations}

\baselineskip 15pt

\section{Introduction}
In this paper, we consider a sort of equations which read as
$$\Delta u(x) = u\tilde{h}(x,\ln u(x)),$$
where $\tilde{h}: \mathbb{R}\times\mathbb{R}\rightarrow\mathbb{R}$ is a continuous function. Obviously, the following equations
\begin{equation}
\Delta u+\lambda_{1}(x)u\ln u +\lambda_{2}(x)u^{b+1}+\lambda_3(x)u^p=0,\label{eq:1.1*}
\end{equation}
defined on a Riemannian manifold $(M, g)$ of dimension $n$, are the special cases of the sort equations. Indeed, we only need to pick
$$\tilde{h}(x, s) = \lambda_{1}(x)se^{s} +\lambda_{2}(x)e^{(b+1)s}+\lambda_3(x)e^{ps},$$
where $\lambda_1, \lambda_2$ and $\lambda_3$ are smooth functions on $M$ and $b, p\in\mathbb{R}$ are two real constant numbers.

In the past four decades, the latter equations with $\lambda_{1}\equiv 0$ include many important and well-known equations stemming from differential geometry and physics etc, and are deeply and extensively studied by many mathematicians. For instance, the works of Schoen and Yau in (\cite{S1, S2, SY}) on conformally flat manifold and Yamabe problem highlighted the importance of studying the distribution solutions of
\begin{equation}
\Delta u(x)+u^{(n+2)/(n-2)}(x)=0.\label{eq:1.2}
\end{equation}
From the viewpoint of analysis, Caffarelli, Gidas and Spruck in \cite{CGS} studied non-negative smooth solutions of the conformal invariant equation ($\ref{eq:1.2}$), and discussed some special form of ($\ref{eq:1.1*}$), written by $\Delta u + g(u)=0$, with an isolated singularity at the origin.

An analogue but more general form of Yamabe's equation is the so-called Einstein-scalar field Lichnerowicz equation. This equation arises from the Hamiltonian constraint equation for the Einstein-scalar field system in general relativity \cite{Lic, York}. In the case the underlying manifold $M$ has dimension $n \geq 3$, the Einstein-scalar field Lichnerowicz equation takes the following form
$$\Delta u + \mu(x)u + A(x)u^p + B(x)u^{-q} = 0,$$
where $\mu(x), A(x)$ and $B(x)$ are smooth functions on $M$ and $p=(n+2)/(n-2)$ and $q=(3n-2)/(n-2)$; while on 2-manifolds, we know that the Einstein-scalar field Lichnerowicz equation is given as follows
$$\Delta u + A(x)e^{2u} + B(x)e^{-2u} + D(x) = 0.$$
Unless otherwise stated, solutions are always required to be smooth and positive. For more details we refer to \cite{C, MW, N} and references therein.

Recently, Peng, Wang and Wei \cite{PWW} used a unified method to consider the gradient estimates of the positive solution to the following nonlinear elliptic equation $$\Delta u + \lambda u^{p}=0$$ defined on a complete noncompact Riemannian manifold $(M, g)$ where $\lambda > 0$ and $ p <1+\frac{4}{n}$ or $\lambda < 0$ and $p >1$ are two constants. For the case $\lambda>0$, their results improve considerably the previous known results and supplements the results for the case $\dim(M)\leq 2$. For the case $\lambda<0$ and $p>1$, they also improved considerably the previous related results. When the Ricci curvature of $(M,g)$ is nonnegative, a Liouville-type theorem for the above equation was established. For more details we refer to \cite{MHL} and references therein.

On the other hand, if we pick $\tilde{h}(x, s) = 2\lambda s e^{s}$ with $\lambda>0$, F. Chung and S.-T. Yau \cite{CY2} showed that if the function $u$ attains the log-Sobolev constant $\lambda_M$ of a closed Riemannian manifold $(M,g)$ with $\dim(M)=n$, then it must satisfy
$$\Delta u+ 2\lambda_M u \ln u=0,$$
and
$$\sup u\leq e^{n/2},\quad \quad \quad  |\nabla u|^2+2\lambda_M u^{2}\ln u\leq \lambda_M n u^{2}.$$
Later, F. Wang \cite{W} extented the results to the case $Ric\geq -K$, and gave a Harnack inequality.

In \cite{M}, Ma investigated the following equation
$$\Delta u + \lambda u \ln u + \mu u = 0$$
on complete non-compact Riemannian manifolds where $\lambda$ and $\mu$ are constant with
$\lambda < 0$, which corresponds to $\tilde{h}(x, s) = \lambda se^{s}+\mu e^s$. His finding for local gradient estimates for positive solutions of this equation is almost optimal if one considers Ricci solitons.

Recently, M. Ghergu, S. Kim and H. Shahgholian in (\cite{GKS}) studied the semilinear elliptic equation
$$\Delta u +u^{\beta}|\ln u|^{\alpha}=0,$$
which corresponds to $\tilde{h}(x, s) = |s|^\alpha e^{\beta s}$, and established that nonnegative solution $u\in C^{2}(B_{1}\backslash {O})$ of the above equation may have a removable singularity at the origin $O$ or behave like some class of functions as $x\rightarrow 0$.

Very recently, Peng, Wang and Wei \cite{PWW1} considered the gradient estimates of the positive solutions to the following equation defined on a complete Riemannian manifold $(M, g)$
$$\Delta u + \lambda u(\ln u)^{p}+ \mu u=0,$$
where $\lambda, \mu\in \mathbb{R}$ and $p$ is a rational number with $p=\frac{k_1}{2k_2+1}\geq2$ where $k_1$ and $k_2$ are positive integer numbers. They obtain the gradient bound of a positive solution to the equation which does not depend on the bounds of the solution and the Laplacian of the distance function on $(M, g)$. Their results can be viewed as a natural extension of Yau's estimates on positive harmonic function.

The parabolic counterpart of the above equation (\ref{eq:1.1*}) was considered by Dung, Khan and Ng\^o \cite{DKN}. More concretely, let $(M, g, e^{-f}dv)$ be a complete, smooth metric measure space with the Bakry-\'Emery Ricci curvature bounded from below, Dung et al have ever studied the following general $f$-heat equations
$$u_t = \Delta_f u + \lambda u\ln u + \mu u + Au^p + Bu^{-q}.$$
Suppose that $\lambda$, $\mu$, $A$, $B$, $p$ and $q$ are constants with $A \leq 0$, $B\geq 0$, $p\geq 1$, and $q \geq 0$. If $u \in (0, 1]$ is a smooth solution to the above general $f$-heat equation, they obtained various gradient estimates for the bounded positive solutions, which depend on the bounds of positive solution and the Laplacian of the distance functions on domain manifolds.

Moreover, they also considered the gradient estimate of bounded positive solution $u\in [1, C)$ to the following equation on a Riemann surface
$$u_t = \Delta_f u + Ae^{2u} + Be^{-2u} + D,$$
where $A$, $B$ and $D$ are constants. Besides, Some mathematicians (see \cite{W, Y-Z}) also paid attention to a similar nonlinear parabolic equation defined on some kind of smooth metric measure space.

In fact, it is of independent interest that one studies various properties of solutions to the following equation
$$\Delta u + u\tilde{h}(x, \ln u)=0$$
defined on a complete Riemannian manifold. In this paper, in order to focus on the core of the problem and not to lengthen this article by adding technicalities, we restrict us to the case $\tilde{h}(x,s)\equiv h(s)$, which is independent of $x$, is a $C^{2}$ function with respect to $s$. Now, the above equations can be written as
$$\Delta u + uh(\ln u)=0.$$
More concretely, first we consider the case of $\lambda_i(x)\equiv constant$ ($i=1, 2$), $\lambda_3\equiv0$ and $b\leq0$. That is, we focus on studying the gradient estimate and the bounds of the positive solution to the following nonlinear elliptic equation defined on an $n$-dimensional complete noncompact Riemannian manifold $(M,g)$
\begin{equation}
\Delta u+\lambda_{1}u\ln u +\lambda_{2}u^{b+1}=0, \label{eq:1.3}
\end{equation}
where $b\leq0$. Then, we turn to studying
\begin{equation}
\Delta u+uh(\ln u)=0\label{eq:1.4}
\end{equation}
where $h$ satisfies some analytic and technical conditions. We try to improve the classical methods to obtain a gradient bound of a positive solution to ($\ref{eq:1.3}$) and ($\ref{eq:1.4}$) which does not depend on the bounds of the solution and the Laplacian or Hessian of the distance function on $(M, g)$.
\medskip

For (\ref{eq:1.3}) we can show the following results:

\begin{theorem}\label{them:1.1}
(Local gradient estimate) Let $(M, g)$ be an $n$-dimensional complete noncompact Riemannian manifold. Suppose there exists a nonnegative constant $K:=K(2R)$ such that the Ricci curvature of $M$ is bounded below by $-K$, i.e., $Ric(g) \geq -Kg$ in the geodesic ball $B_{2R}(O)\subset M$ where $O$ is a fixed point on $M$. Suppose that $u(x)$ is a smooth positive solution to equation (\ref{eq:1.3}) on $B_{2R}(O)$ with $b\leq0$.

Case 1: If $\lambda_{1}>0$ and $\lambda_{2}>0$, then, for any constant $1<p<2$ there holds true on $B_{R}(O)$\label{eq:1.5}
\begin{equation}
\frac{|\nabla u|^2}{u^2}+p \lambda_{1}\ln u+\lambda_{2}u^{b}\leq \tilde{C_{1}}(n, K, R, \lambda_{1}, \lambda_{2}, b, p);
\end{equation}
where
\begin{equation}
\tilde{C_{1}}(n, K, R, \lambda_{1}, \lambda_{2}, b, p)= \max\left\{
\begin{aligned}
&\min_{C_3\in(0,\frac{2(2-p)}{np})}\left\{\frac{((A+2K+2\lambda_{1})R^{2}C_{3}+C_{1}^{2})np}{(2(2-p)-C_{3}np)R^{2}C_{3}}\right\};\\
& nA+\frac{n^{2}C_{1}^{2}}{R^{2}}+2Kn+n(p-2)\lambda_{1}+np\lambda_{1};\\
& \frac{n}{2(p-1)}\left(\frac{2}{n}(p-1)^{2}L+p\lambda_{1}+2pK \right)
\end{aligned}\label{eq:1.6}
\right\}
\end{equation}
with $C_{1}$ and ${C_{2}}$ are absolute constants independent of the geometry of $M$. Here
$$A=\frac{((n-1)(1+\sqrt{K} R)+2)C_{1}^{2}+C_2}{R^{2}}\quad\quad\mbox{and}\quad\quad L=\frac{n(p\lambda_{1}+2pK)}{2(p-1)^2}.$$

Case 2: If $\lambda_{1}\leq 0$ and $\lambda_{2}>0$, then, for any constant $1<p<2$ there holds true on $B_{R}(O)$
\begin{equation}
\frac{|\nabla u|^2}{u^2}+p \lambda_{1}\ln u+\lambda_{2}u^{b}\leq \tilde{C_{2}}(n, K, R, \lambda_{1}, \lambda_{2}, b, p);\label{eq:1.7}
\end{equation}
where
\begin{equation}
\tilde{C_{2}}(n, K, R, \lambda_{1}, \lambda_{2}, b, p)= \max\left\{
\begin{aligned}
\frac{nKp}{2(p-1)};
\min_{C_3\in(0,\frac{2(2-p)}{np})}\left\{\frac{((A+2K)R^{2}C_{3}+C_{1}^{2})np}{(2(2-p)-C_{3}np)R^{2}C_{3}}\right\}
\end{aligned}\label{eq:1.8}
\right\}
\end{equation}
with the same $C_{1}$, ${C_{2}}$ and $A$ as in Case 1.
\end{theorem}

For the equation (\ref{eq:1.3}) with $b<0$, we can see easily from the estimates in Theorem \ref{them:1.1} that the following consequences hold true as a direct corollary.
\begin{cor}\label{cor:1.0}
Let $(M, g)$ be an $n$-dimensional complete noncompact Riemannian manifold. Suppose there exists a nonnegative constant $K$ such that the Ricci curvature of $M$ is bounded below by $-K$, i.e., $Ric(g) \geq -Kg$ in $M$. Suppose that the equation (\ref{eq:1.3}) defined on $M$ with $b<0$ admits a smooth positive solution $u(x)$. Then, the positive solution $u$ is of upper bound and positive lower bound if $\lambda_{1}>0$ and $\lambda_{2}>0$, and is of a positive lower bound if $\lambda_{1}\leq 0$ and $\lambda_{2}>0$.
\end{cor}

\begin{remark}
In comparison with the relative results stated in Theorem 1.1 of \cite{DKN}, we does not need to assume the positive solutions are bounded to derive the gradient estimates. Moreover, we can also establish the estimates on the bounds of positive solutions in some cases, which does not depend the Laplacian or Hessian of the distance function on $(M,g)$.
\end{remark}

\begin{remark}
For the equation (\ref{eq:1.1*}) with $\lambda_1(x)$, $\lambda_2(x)>0$ and $\lambda_3(x)<0$ are three real constant numbers and $p>0$, we can also employ the same method to establish some similar estimates for the equation (\ref{eq:1.1*}) with the above results, we will present them in forthcoming papers.
\end{remark}

For more general equation ($\ref{eq:1.4}$) we can also obtain some results by a delicate analysis. Now we state them as follows:

\begin{theorem}\label{them:1.2}
(Local gradient estimate) Let $(M, g)$ be an n-dimensional complete noncompact Riemannian manifold. Suppose there exists a nonnegative constant K:=K(2R) such that
the Ricci Curvature of M is bounded below by $-K$, i.e., $Ric(g) \geq -Kg$ in the geodesic ball $B_{2R}(O)\subset M$ where O is a fixed point on M. Suppose $h\in C^2(\mathbb{R})$ and there exist a $\lambda$ such that
\begin{equation}
\left\{
\begin{aligned}
& -\frac{4}{n}(\lambda-1)h+(\lambda-2)h'+\lambda h''\geq 0;\\
& h(2K\lambda-\frac{2}{n}(\lambda^{2}-1)h-\lambda h')\geq 0;\\
& \lambda h\geq 0.
\end{aligned}
\right.\label{eq:1.9}
\end{equation}
If $u(x)$ is a smooth positive solution to equation (\ref{eq:1.4}) on $B_{2R}(O)$, then we have
\begin{equation}
\frac{|\nabla u|^2}{u^2}+\lambda h(\ln u)\leq C(n, K, R, h)\label{eq:1.10}
\end{equation}
and
\begin{equation}
\frac{|\nabla u|}{u}\leq \sqrt{C(n, K, R, h)},\label{eq:1.11}
\end{equation}
where
\begin{equation}
C(n, K, R, h)=\min_{C_{5}\in (0, 2/n)}\left\{\frac{((n-1)(1+\sqrt{K}R)+2+\frac{1}{C_{5}})C_{1}^{2}+2KR^{2}}{R^2 (\frac{2}{n}-C_{5})}\right\}.\label{eq:1.12}
\end{equation}
Here $C_{1}$ is an absolute constant independent of the geometry of $M$.
\end{theorem}

It is well-known that, for any two points $x,\, y\in{B_{R/2}(O)}$, there holds true
$$\ln u(x)-\ln u(y)\leq\int_{\gamma}{\frac{|\nabla u|}{u}},$$
where $\gamma$ is a curve connecting $x$ and $y$ in $M$. It follows
\begin{cor}\label{cor:1.1}
(Harnack inequality) Suppose the same conditions as in Theorem $\ref{them:1.2}$ hold. Then
$$\sup_{B_{R/2}(O)}u\leq e^{R\sqrt{C(n, K, R, h)}} \inf_{B_{R/2}(O)}u.$$
\end{cor}

When $K=0$, letting $R\rightarrow +\infty$ in Theorem $\ref{them:1.2}$, then we have
\begin{cor}\label{cor:1.2}
(Liouville-type result) Let $(M, g)$ be an n-dimensional noncompact complete Riemannian manifold with nonnegative Ricci curvature. Suppose the same conditions as in Theorem $\ref{them:1.2}$ hold. Then any positive solution $u$ of $(\ref{eq:1.4})$ must be constant. Moreover, if $\lambda h > 0$, $(\ref{eq:1.4})$ admits no positive solutions.
\end{cor}

It is convenient that we find some sufficient conditions on function $h$ to ensure the effectiveness of the method adopted here and make the conditions of Theorem $\ref{them:1.2}$ satisfy. By a direct calculation we can see easily that, if $h(\ln u)\geq 0$, $h'(\ln u)\leq 0$, $h''(\ln u)\geq 0$, and $0 \leq \lambda \leq 1$, then (\ref{eq:1.9}) holds true. These are some sufficient conditions to guarantee the assumptions in Theorem $\ref{them:1.2}$ are satisfied. Hence
\begin{cor}\label{cor:1.3}
If $h(\ln u)\geq 0$, $h'(\ln u)\leq 0$, $h''(\ln u)\geq 0$, and $0 \leq \lambda \leq 1$, then (\ref{eq:1.9}) holds true, therefore, the conclusions of Theorem $\ref{them:1.2}$ hold true.
\end{cor}
In other words, a decreasing, convex and smooth nonnegative function $h(x)$ satisfies the assumptions in Corollary \ref{cor:1.3}.

\begin{example}
(i) It is easy to see that equation $$\Delta u +cu^{d+1}=0$$
satisfies the above sufficient conditions. Here $h(\ln u)=cu^{d}$ with constants $c\geq0$ and $d\leq0$.
Direct calculation shows that $h=cu^{d}\geq0$, $h'=cdu^{d}\leq0$ and $h''=cd^{2}u^{d}\geq0$.\\
(ii) Obviously, $$\Delta u +\Sigma_{i=1}^{n}c_{i}u^{d_{i}+1}=0$$
satisfies the sufficient conditions with $c_{i}\geq0$ and $d_{i}\leq 0$.
\end{example}

\begin{example}
For equation $\Delta u-u^{3}=0$ , the corresponding function $h$ satisfies $h(\ln u)=-u^{2}=-e^{2\ln u}<0$. Choosing $\lambda =0$ and taking a direct calculation we will see that there holds true
$$-\frac{4}{n}(\lambda-1)h+(\lambda-2)h'+\lambda h''=4\left(1-\frac{1}{n}\right)e^{2\ln u}\geq0,$$ and
$$h(2K\lambda-\frac{2}{n}(\lambda^{2}-1)h-\lambda h')=\frac{2}{n}e^{4\ln u}\geq0.$$
So, this example satisfies the conditions (\ref{eq:1.9}) supposed in Theorem $\ref{them:1.2}$, but doesn't satisfy the above sufficient conditions.
\end{example}

For the case $\lambda=1$, we take the same argument as in Theorem $\ref{them:1.2}$ to conclude the following:
\begin{cor}\label{cor:1.4}
Let $(M, g)$ be an n-dimensional complete noncompact Riemannian manifold. Suppose there exists a nonnegative constant K:=K(2R) such that
the Ricci Curvature of M is bounded below by $-K$, i.e., $Ric(g) \geq -Kg$ in the geodesic ball $B_{2R}(O)\subset M$ where O is a fixed point on M. Suppose that $u(x)$ is a smooth
positive solution to equation (\ref{eq:1.4}). If $h$ satisfies that $h'(\ln u)\leq \min\{h''(\ln u),2K\}$ and $h(\ln u)\geq 0$ on $B_{R}(O)$, then
\begin{equation}
\frac{|\nabla u|^2}{u^2}+h\leq C(n, K, R, h), \label{eq:1.13}
\end{equation}
and
\begin{equation}
\frac{|\nabla u|}{u}\leq \sqrt{C(n, K, R, h)}.\label{eq:1.14}
\end{equation}
Here, $C(n, K, R, h)$ is the same as in Theorem 1.2.
\end{cor}

On the other hand, by taking the same discussion as in the proof of Theorem $\ref{them:1.2}$ for the case $\lambda=0$ we can also conclude the following
\begin{cor}\label{cor:1.5}
Let $(M, g)$ be an n-dimensional complete noncompact Riemannian manifold. Suppose there exists a nonnegative constant K:=K(2R) such that
the Ricci Curvature of M is bounded below by $-K$, i.e., $Ric(g) \geq -Kg$ in the geodesic ball $B_{2R}(O)\subset M$ where O is a fixed point on M. Suppose that $u(x)$ is a smooth positive solution to equation (\ref{eq:1.4}). If the function $h$ satisfied $h'(\ln u)\leq\frac{2}{n}h(\ln u)$ on $B_{R}(O)$, then
\begin{equation}
\frac{|\nabla u|}{u}\leq \sqrt{C(n, K, R, h)}.\label{eq:1.16}
\end{equation}
Here, $C(n, K, R, h)$ is the same as in Theorem 1.2.
\end{cor}

It is worthy to point out that any positive function $h(\ln u)$, which is decreasing with $\ln u$, satisfies the assumptions posed in Corollary \ref{cor:1.5}. So, there are too many choices.
\begin{example}
For instance, the equation reads
$$\Delta u+au\left(\frac{\pi}{2}-\arctan (\ln u)\right)=0$$
with constant $a>0$, then, we can verify easily that there hold
$$h(\ln u)=\frac{a\pi}{2}-a\arctan (\ln u)\geq 0\quad\quad \mbox{and} \quad\quad h'(\ln u)=-\frac{a}{1+(\ln u)^{2}}<0.$$
\end{example}

In the forthcoming paper we will discuss the equation (\ref{eq:1.4}) including Lichnerowicz equation as special case, and some more general equations than those in \cite{DKN}.

Finally, we would like to mention that the strategy of our proofs follows basically those in \cite{LY, PWW, PWW1}. More precisely, we use an appropriate cut-off function and the maximum principle to obtain the desired results. These methods are, loosely speaking, well-known and used in many works; for instance, see \cite{DKN, LY, MHL, PWW, W1} and the references therein. However, we also would like to emphasize that to obtain gradient estimates of these equations discussed here, our approach is slightly different from those used before.  Except for apply the Bochner-Weitzenb\"ock formula to a suitable auxiliary function $G$ related to $\ln u$ ( see Section 2), we need to analyze carefully the equation which is satisfied by $G$ and estimate delicately all terms appeared so that the required terms do match very well. Then we make use of the maximum principle to prove our results.

The paper is organized as follows. In Section 2, we recall some notations and fundamental lemmas. In Section 3 we provide the proof of Theorem 1.1. In Section 4 we study gradient estimates of the general equation $(\ref{eq:1.4})$ and prove Theorem 1.2. Harnack-type inequalities and Liouville-type theorems for $(\ref{eq:1.4})$ are also established in this section.

\section{Preliminaries}

In this section, we denote $(M, g)$ an $n$-dimensional complete Riemannian manifold with $Ric(g)\geq-Kg$ in the geodesic ball $B_{2R}(O)$, where $K=K(2R)$ is a nonnegative constant depending on $R$ and $O$ is a fixed point on $M$.

It is easy to deformed the equation ($\ref{eq:1.3}$) into
$$\Delta u+uf(\ln u)+ug(\ln u)=0,$$
where $f,g\in C^{2}(\mathbb{R}, \mathbb{R})$ are $C^{2}$ functions.

For ($\ref{eq:1.3}$), Bo Peng, Youde Wang and Guodong Wei have ever proved an important inequality in \cite{PWW, PWW1}:

\begin{prop}
Suppose that $u(x)$ is a smooth positive solution to equation $(\ref{eq:1.3})$ on $B_{2R}(O)$. Let
$$\omega =\ln u \quad\quad \mbox{and}\quad \quad  G=|\nabla{\omega}|^{2}+\beta_{1} f(\omega)+\beta_{2}  g(\omega),$$
here $\beta_{1}$ and $\beta_{2}$ are constants to be determined later.
 Then we have
\begin{equation}
\begin{aligned}
\Delta G\geq & \frac{2}{n}(G-(\beta_{1}-\lambda_{1})f-(\beta_{2}-\lambda_{2})g)^{2}-2\langle\nabla\omega, \nabla G\rangle\\
&+((\beta_{1}-2\lambda_{1})f'+\beta_{1}f''+(\beta_{2}-2\lambda_{2})g'+\beta_{2}g''-2K)(G-\beta_{1}f-\beta_{2}g) \\
&-(\beta_{1}f'+\beta_{2}g')(G-(\beta_{1}-\lambda_{1})f-(\beta_{2}-\lambda_{2})g). \label{eq:2.1}
\end{aligned}
\end{equation}
\end{prop}

\begin{proof}
First, there holds
\begin{equation}
\Delta\omega+G-(\beta_{1}-\lambda_{1})f-(\beta_{2}-\lambda_{2})g=0
\label{eq:2.2}
\end{equation}
and
\begin{equation}
|\nabla\omega|^{2}=G-\beta_{1}f-\beta_{2}g.\label{eq:2.3}
\end{equation}
By the Bochner-Weitzenb$\ddot{o}$ck's formula and $Ric(g)\geq-Kg$ on $(M,g)$, we obtain

\begin{equation}
\Delta|\nabla\omega|^{2}\geq2|\nabla^{2}\omega|^{2}+2\langle\nabla\omega,\nabla(\Delta\omega)\rangle-2K|\nabla\omega|^2.\label{eq:2.4}
\end{equation}
Combining (\ref{eq:2.2}), (\ref{eq:2.3}) and (\ref{eq:2.4}), we obtain
\begin{equation}
\begin{aligned}
\Delta G=&\Delta|\nabla\omega|^{2}+\Delta(\beta_{1}f+\beta_{2}g)\\
\geq &2|\nabla^{2}\omega|^{2}+2\langle\nabla\omega,\nabla(\Delta\omega)\rangle-2K|\nabla\omega|^2+\Delta(\beta_{1}f+\beta_{2}g)\\
\geq &\frac{2}{n}(\Delta \omega)^{2}+2\langle\nabla\omega,\nabla(\Delta\omega)\rangle-2K|\nabla\omega|^2+\beta_{1}(f''|\nabla\omega|^{2}+f'\Delta\omega)+\beta_{2}(g''|\nabla\omega|^{2}+g'\Delta\omega).
\label{eq:2.5}
\end{aligned}
\end{equation}
Here we have used the relation $$|\nabla^{2}\omega|^{2}\geq \frac{2}{n}(\Delta \omega)^{2},$$
which is derived by Cauchy-Schwarz inequality. Substituting ($\ref{eq:2.2}$) and ($\ref{eq:2.3}$) into ($\ref{eq:2.5}$), we obtain
\begin{equation}
\begin{aligned}
\Delta G\geq &\frac{2}{n}(G-(\beta_{1}-\lambda_{1})f-(\beta_{2}-\lambda_{2})g)^{2}-2\langle\nabla\omega,\nabla(G-(\beta_{1}-\lambda_{1})f-(\beta_{2}-\lambda_{2})g)\rangle\\
&-2K|\nabla\omega|^2+\beta_{1}(f''|\nabla\omega|^{2}+f'\Delta\omega)+\beta_{2}(g''|\nabla\omega|^{2}+g'\Delta\omega)\\
 =& \frac{2}{n}(G-(\beta_{1}-\lambda_{1})f-(\beta_{2}-\lambda_{2})g)^{2}-2\langle\nabla\omega, \nabla G\rangle\\
&+((\beta_{1}-2\lambda_{1})f'+\beta_{1}f''+(\beta_{2}-2\lambda_{2})g'+\beta_{2}g''-2K)|\nabla\omega|^{2} \\
&-(\beta_{1}f'+\beta_{2}g')(G-(\beta_{1}-\lambda_{1})f-(\beta_{2}-\lambda_{2})g)\\
=& \frac{2}{n}(G-(\beta_{1}-\lambda_{1})f-(\beta_{2}-\lambda_{2})g)^{2}-2\langle\nabla\omega, \nabla G\rangle\\
&+((\beta_{1}-2\lambda_{1})f'+\beta_{1}f''+(\beta_{2}-2\lambda_{2})g'+\beta_{2}g''-2K)(G-\beta_{1}f-\beta_{2}g) \\
&-(\beta_{1}f'+\beta_{2}g')(G-(\beta_{1}-\lambda_{1})f-(\beta_{2}-\lambda_{2})g).
\end{aligned}\label{eq:2.6}
\end{equation}
Thus we complete the proof.
\end{proof}

Let $\phi(x)$  be a $C^{2}$ cut-off function with $0\leq \phi(x)\leq 1$, $\phi(x)|_{B_{R}(O)}=1$ and $\phi(x)|_{M\backslash B_{2R}(O)}=0$. Using Laplacian comparison theorem (see \cite{LY}), there holds true
\begin{equation}
\frac{|\nabla\phi|^{2}}{\phi}\leq \frac{C_{1}^{2}}{R^{2}},
\quad\quad \mbox{and} \quad\quad
\Delta \phi\geq-\frac{(n-1)(1+\sqrt{K}RC_{1}^{2}+C_2)}{R^{2}},\label{eq:2.7}
\end{equation}
where $C_{1}$ and $C_{2}$ are absolute constants.

Take $x_{0}\in B_{2R}(O)$ such that
$$\phi G(x_{0})=\sup_{B_{2R}(O)}(\phi G)\geq 0.$$
Otherwise, if $$\sup_{B_{2R}(O)}(\phi G)<0,$$
the conclusion is trivial.

Since $x_{0}$ is a maximum point of $\phi G$ on $B_{2R}(O)$, at $x_{0}$ we have
$$\nabla(\phi G)=0\quad\quad \mbox{and} \quad\quad \Delta(\phi G)\leq 0.$$
That is
\begin{equation}
\phi \nabla G=-G\nabla\phi \quad\quad \mbox{and} \quad\quad \phi\Delta G\leq -G\Delta\phi+2G\frac{|\nabla\phi|^{2}}{\phi}.\label{eq:2.8}
\end{equation}

In the sequel, for the sake of convenience we do neglect $x_0$. Setting
$$AG:=\frac{((n-1)(1+\sqrt{K} R)+2)C_{1}^{2}+C_2}{R^{2}}G,$$
we can see easily that
$$AG\geq-G\Delta\phi+2G\frac{|\nabla\phi|^{2}}{\phi}\geq \phi\Delta G.$$
Now, from ($\ref{eq:2.6}$) we obtain
\begin{equation*}
\begin{aligned}
AG\geq\phi\Delta G \geq & \frac{2}{n}(G-(\beta_{1}-\lambda_{1})f-(\beta_{2}-\lambda_{2})g)^{2}\phi-2\langle\nabla\omega, \nabla G\rangle\phi\\
&+((\beta_{1}-2\lambda_{1})f'+\beta_{1}f''+(\beta_{2}-2\lambda_{2})g'+\beta_{2}g''-2K)(G-\beta_{1}f-\beta_{2}g)\phi \\
&-(\beta_{1}f'+\beta_{2}g')(G-(\beta_{1}-\lambda_{1})f-(\beta_{2}-\lambda_{2})g)\phi.
\end{aligned}
\end{equation*}
Noticing
$$-2\langle\nabla\omega, \nabla G\rangle\phi=2\langle\nabla\omega, \nabla \phi\rangle G \geq-2|\nabla\omega|| \nabla \phi| G=-2| \nabla \phi| G(G-\beta_{1}f-\beta_{2}g)^{\frac{1}{2}},$$
we obtain
\begin{equation}
\begin{aligned}
AG\geq & \frac{2}{n}(G-(\beta_{1}-\lambda_{1})f-(\beta_{2}-\lambda_{2})g)^{2}\phi-2| \nabla \phi| G(G-\beta_{1}f-\beta_{2}g)^{\frac{1}{2}}\\
&+((\beta_{1}-2\lambda_{1})f'+\beta_{1}f''+(\beta_{2}-2\lambda_{2})g'+\beta_{2}g''-2K)(G-\beta_{1}f-\beta_{2}g)\phi \\
&-(\beta_{1}f'+\beta_{2}g')(G-(\beta_{1}-\lambda_{1})f-(\beta_{2}-\lambda_{2})g)\phi.\label{eq:2.9}
\end{aligned}
\end{equation}
Now we are ready to provide the proofs of these Theorems.

\section{The proof of Theorem 1.1}
In this section, we consider the gradient estimates of (\ref{eq:1.3}), i.e.
\begin{equation*}
\Delta u+\lambda_{1}u\ln u +\lambda_{2}u^{b+1}=0,
\end{equation*}
where $b\leq0$.
Now, we present the proof of Theorem 1.1.
\begin{proof} Letting $\beta_{1}=p\lambda_{1}$, $\beta_{2}=q\lambda_{2}$, $f(\omega)=\omega$ and $g(\omega)=e^{b\omega}$ in $(\ref{eq:2.9})$, we know that at $x_{0}$ there holds true
\begin{equation}
\begin{aligned}
AG\geq & \frac{2}{n}(G-(p-1)\lambda_{1}f-(q-1)\lambda_{2}g)^{2}\phi-2| \nabla \phi| G(G-p\lambda_{1}f-q\lambda_{2}g)^{\frac{1}{2}}\\
&+((p-2)\lambda_{1}f'+p\lambda_{1}f''+(q-2)\lambda_{2}g'+q\lambda_{2}g''-2K)
(G-p\lambda_{1}f-q\lambda_{2}g)\phi \\
&-(p\lambda_{1}f'+q\lambda_{2}g')(G-(p-1)\lambda_{1}f-(q-1)\lambda_{2}g)\phi\\
=& \frac{2}{n}(G-(p-1)\lambda_{1}\omega-(q-1)\lambda_{2}e^{b\omega})^{2}\phi-2| \nabla \phi| G(G-p\lambda_{1}\omega-q\lambda_{2}e^{b\omega})^{\frac{1}{2}}\\
&+((p-2)\lambda_{1}+(q-2)\lambda_{2}be^{b\omega}+q\lambda_{2}b^{2}e^{b\omega}-2K)(G-p\lambda_{1}\omega-q\lambda_{2}e^{b\omega})\phi \\
&-(p\lambda_{1}+q\lambda_{2}be^{b\omega})
(G-(p-1)\lambda_{1}\omega-(q-1)\lambda_{2}e^{b\omega})\phi.\label{eq:3.1}
\end{aligned}
\end{equation}
Letting $1<p<2$ and $q=1$ in ($\ref{eq:3.1}$), we have
\begin{equation}
\begin{aligned}
AG\geq & \frac{2}{n}(G-(p-1)\lambda_{1}\omega)^{2}\phi-2| \nabla \phi| G(G-p\lambda_{1}\omega-\lambda_{2}e^{b\omega})^{\frac{1}{2}}\\
&+((p-2)\lambda_{1}-\lambda_{2}be^{b\omega}+\lambda_{2}b^{2}e^{b\omega}-2K)(G-p\lambda_{1}\omega-\lambda_{2}e^{b\omega})\phi \\
&-(p\lambda_{1}+\lambda_{2}be^{b\omega})
(G-(p-1)\lambda_{1}\omega)\phi.\label{eq:3.2}
\end{aligned}
\end{equation}

Case 1: $\lambda_{1}>0$ and $\lambda_{2}>0$. In order to obtain the required estimates we need to treat each term appeared in the above inequality by a delicate way. Therefore, we need to set the following positive number
$$L=\frac{n(p\lambda_{1}+2Kp)}{2(p-1)^2}$$
and divide the value range of $\omega$ into three intervals:
(1). $\omega\geq 0$;
(2). $-L<\omega<0$;
(3). $\omega\leq -L$.
Then, according to the value range intervals of $\omega$ we will deal with the above inequality (\ref{eq:3.2}) carefully one by one.

(1). $\omega\geq 0$.

Using Young's inequality, we can deduce that there holds
\begin{equation}
2G|\nabla\phi|(G-p\lambda_{1}\omega-\lambda_{2}e^{b\omega})^{1/2}\leq C_{3}\phi G(G-p\lambda_{1}\omega-\lambda_{2}e^{b\omega})+\frac{|\nabla\phi|^{2}}{\phi}\frac{G}{C_{3}},
\label{eq:3.3}
\end{equation}
where $C_{3}$ is a positive constant to be determined later.

Noticing $b\leq 0$ and using ($\ref{eq:2.7}$) and ($\ref{eq:3.2}$) we have
\begin{equation}
\begin{aligned}
AG
\geq& \frac{2}{n}(G-(p-1)\lambda_{1}\omega)^{2}\phi-\frac{C_{1}^{2}}{R^{2}}\frac{G}{C_{3}}-C_{3}\phi G^{2} +C_{3}\phi G(p\lambda_{1}\omega+\lambda_{2}e^{b\omega}) \\
&+((p-2)\lambda_{1}-\lambda_{2}be^{b\omega}+\lambda_{2}b^{2}e^{b\omega}-2K)(G-p\lambda_{1}\omega-\lambda_{2}e^{b\omega})\phi \\
&-(p\lambda_{1}+\lambda_{2}be^{b\omega})
(G-(p-1)\lambda_{1}\omega)\phi\\
\geq &\frac{2}{n}\phi G^{2}-\frac{4}{n}\phi G(p-1)\lambda_{1}\omega-\frac{C_{1}^{2}}{R^{2}}\frac{G}{C_{3}}-C_{3}\phi G^{2} +C_{3}\phi G(p\lambda_{1}\omega+\lambda_{2}e^{b\omega}) \\
&+\phi\frac{2}{n}(p-1)^{2}\lambda_{1}^2\omega^{2}+((p-2)\lambda_{1}-2K)(G-p\lambda_{1}\omega-\lambda_{2}e^{b\omega})\phi \\
&-(p\lambda_{1}+\lambda_{2}be^{b\omega})
(G-(p-1)\lambda_{1}\omega)\phi.
\end{aligned}\label{eq:3.4}
\end{equation}
After rearranging the right side of the above inequality, we have
\begin{equation}
\begin{aligned}
AG
\geq &\frac{2}{n}\phi G^{2}-C_{3}\phi G^{2}-\frac{4}{n}\phi G(p-1)\lambda_{1}\omega-\frac{C_{1}^{2}}{R^{2}}\frac{G}{C_{3}}-2KG\phi\\
&+(p-2)\lambda_{1}G\phi-p\lambda_{1}G\phi+\lambda_{2}e^{b\omega}\phi(C_{3}G-bG+b(p-1)\lambda_{1}\omega)\\
&+\lambda_{1}^2\phi(\frac{2}{n}(p-1)^{2}\omega^{2}+p(2-p)\omega+p(p-1)\omega)+2Kp\lambda_{1}\omega \phi.\\
\label{eq:3.5}
\end{aligned}
\end{equation}

Noticing that
$$K\geq0, \quad\quad 0\leq\phi\leq 1, \quad\quad 0\leq p\lambda_{1}\omega\leq G \quad\quad\mbox{and}\quad\quad 0\leq \lambda_{2}e^{b\omega}\leq G,$$
we have
\begin{equation}
\begin{aligned}
AG\geq & \left(\frac{2}{n}-C_{3}-\frac{4(p-1)}{np}\right)G^{2}\phi-\frac{C_{1}^{2}}{R^{2}}\frac{G}{C_{3}}-2K\phi G-2\lambda_{1}G\phi\\
&+\lambda_{2}e^{b\omega}\phi G\left(C_{3}-b+\frac{b(p-1)}{p}\right)\\
\geq & \left(\frac{2(2-p)}{np}-C_{3}\right)G^{2}\phi-\frac{C_{1}^{2}}{R^{2}}\frac{G}{C_{3}}-(2K+2\lambda_{1})\phi G.\\
\label{eq:3.6}
\end{aligned}
\end{equation}
If $$\frac{2(2-p)}{np}-C_{3}>0,$$
then, dividing the both sides of the above inequality by $G$ we obtain
\begin{equation}
\begin{aligned}
A\geq & (\frac{2(2-p)}{np}-C_{3})G\phi-\frac{C_{1}^{2}}{R^{2}}\frac{1}{C_{3}}-(2K+2\lambda_{1})\phi. \end{aligned}\label{eq:3.7}
\end{equation}
Thus, we know that for all $C_{3}\in(0,\frac{2(2-p)}{np})$ there holds true
\begin{equation}
\begin{aligned}
\sup_{B_{R}(O)}G
\leq G\phi
\leq&\frac{((A+2K+2\lambda_{1})R^{2}C_{3}+C_{1}^{2})np}{(2(2-p)-C_{3}np)R^{2}C_{3}},\label{eq:3.8}
\end{aligned}
\end{equation}
where
$$A=\frac{((n-1)(1+\sqrt{K} R)+2)C_{1}^{2}+C_2}{R^{2}}. $$

On the other hand, we note that the right hand side of the above inequality tends to $+\infty$ if $C_{3}\rightarrow 0^{+}$ or $C_{3}\rightarrow (\frac{2(2-p)}{np})^{-}$. The right hand side of (\ref{eq:3.8}) is a continuous function of variable $C_{3}$, thus it can take its minimum in $(0,\frac{2(2-p)}{np})$.

\medskip

(2). $-L<\omega<0$.

In the present situation, by using Young's inequality we can verify that there holds
\begin{equation}
2G|\nabla\phi|(G-p\lambda_{1}\omega-\lambda_{2}e^{b\omega})^{1/2}\leq C_{4}\phi^{1/2}G^{1/2} (G-p\lambda_{1}\omega-\lambda_{2}e^{b\omega})+\frac{|\nabla\phi|^{2}}{C_{4}\phi}\phi^{1/2}G^{3/2},\label{eq:3.9}
\end{equation}
where $C_{4}$ is a positive constant to be determined later.

From ($\ref{eq:3.2}$) we have
\begin{equation}
\begin{aligned}
AG
\geq& \frac{2}{n}(G-(p-1)\lambda_{1}\omega)^{2}\phi-C_{4}\phi^{1/2}G^{1/2} (G-p\lambda_{1}\omega-\lambda_{2}e^{b\omega})-\frac{|\nabla\phi|^{2}}{C_{4}\phi}\phi^{1/2}G^{3/2}
 \\
&+((p-2)\lambda_{1}-\lambda_{2}be^{b\omega}+\lambda_{2}b^{2}e^{b\omega}-2K)(G-p\lambda_{1}\omega-\lambda_{2}e^{b\omega})\phi \\
&-(p\lambda_{1}+\lambda_{2}be^{b\omega})
(G-(p-1)\lambda_{1}\omega)\phi\\
\geq &\frac{2}{n}\phi G^{2}-\frac{4}{n}\phi G(p-1)\lambda_{1}\omega-C_{4}\phi^{1/2}G^{1/2} (G-p\lambda_{1}\omega-\lambda_{2}e^{b\omega})-\frac{|\nabla\phi|^{2}}{C_{4}\phi}\phi^{1/2}G^{3/2} \\
&+\phi\frac{2}{n}(p-1)^{2}\lambda_{1}^2\omega^{2}+((p-2)\lambda_{1}-2K)(G-p\lambda_{1}\omega-\lambda_{2}e^{b\omega})\phi \\
&-(p\lambda_{1}+\lambda_{2}be^{b\omega})
(G-(p-1)\lambda_{1}\omega)\phi.\label{eq:3.10}
\end{aligned}
\end{equation}
By taking an rearrangement of the terms on the right side of the above inequality, we have
\begin{equation}
\begin{aligned}
AG
\geq &\frac{2}{n}\phi G^{2}-C_{4}\phi^{1/2}G^{3/2}-\frac{C_{1}^{2}}{C_{4}R^{2}}\phi^{1/2}G^{3/2}\\
&-2KG\phi+(p-2)\lambda_{1}G\phi-p\lambda_{1}G\phi\\
&+\lambda_{2}e^{b\omega}\phi(C_{4}\phi^{1/2}G^{1/2}-bG+b(p-1)\lambda_{1}\omega)\\
&+\lambda_{1}\omega(-\frac{4}{n}\phi G(p-1)+C_{4}p\phi^{1/2}G^{1/2}+\frac{2}{n}(p-1)^{2}\omega\phi+p\lambda_{1}\phi+2Kp \phi).\label{eq:3.11}
\end{aligned}
\end{equation}
If
$$0>-\frac{4}{n}\phi G(p-1)+C_{4}p\phi^{1/2}G^{1/2}+\frac{2}{n}(p-1)^{2}\omega\phi+p\lambda_{1}\phi+2Kp \phi,$$
then, from $(\ref{eq:3.11})$ we infer
\begin{equation}
\begin{aligned}
AG
\geq&\frac{2}{n}\phi G^{2}-C_{4}\phi^{1/2}G^{3/2}-\frac{C_{1}^{2}}{C_{4}R^{2}}\phi^{1/2}G^{3/2}\\
&-2KG\phi-2\lambda_{1}G\phi.\label{eq:3.12}
\end{aligned}
\end{equation}
Dividing the both sides of the above inequality by $G$, we obtain
\begin{equation}
\begin{aligned}
A\geq&\frac{2}{n}\phi G-C_{4}\phi^{1/2}G^{1/2}-\frac{C_{1}^{2}}{C_{4}R^{2}}\phi^{1/2}G^{1/2}-2K\phi-2\lambda_{1}\phi\\
\geq&\frac{1}{n}\phi G-\frac{n}{4}(C_{4}+\frac{C_{1}^{2}}{C_{4}R^{2}})^{2}-2K\phi-2\lambda_{1}\phi.\label{eq:3.13}
\end{aligned}
\end{equation}
Thus, at $x_{0}$ we have
\begin{equation}
\begin{aligned}
\phi G\leq & \inf_{C_{4}>0}\{nA+\frac{n^{2}}{4}(C_{4}+\frac{C_{1}^{2}}{C_{4}R^{2}})^{2}+2Kn+n(p-2)\lambda_{1}+np\lambda_{1}\}\\
=& nA+\frac{n^{2}C_{1}^{2}}{R^{2}}+2Kn+n(p-2)\lambda_{1}+np\lambda_{1},\label{eq:3.14}
\end{aligned}
\end{equation}
where $$A=\frac{((n-1)(1+\sqrt{K} R)+2)C_{1}^{2}+C_2}{R^{2}}.$$
Otherwise, we have
$$0\leq-\frac{4}{n}\phi G(p-1)+C_{4}p\phi^{1/2}G^{1/2}+\frac{2}{n}(p-1)^{2}\omega\phi+p\lambda_{1}\phi+2Kp \phi,$$
then, it follows that
$$0\leq-\frac{2}{n}\phi G(p-1)+\frac{n}{8(p-1)}C_{4}^{2}p^{2}+\frac{2}{n}(p-1)^{2}L\phi+p\lambda_{1}\phi+2Kp \phi.$$
This leads to
\begin{equation}
\begin{aligned}
\phi G\leq & \inf_{C_{4}}\left\{\frac{n}{2(p-1)}(\frac{n}{8(p-1)}C_{4}^{2}p^{2}+\frac{2}{n}(p-1)^{2}L+p\lambda_{1}+2Kp )\right\}\\
=& \frac{n}{2(p-1)}\left(\frac{2}{n}(p-1)^{2}L+p\lambda_{1}+2Kp \right).\label{eq:3.15}
\end{aligned}
\end{equation}
\medskip

(3). $\omega\leq-L$.

For this case, we have
$$\frac{2}{n}(p-1)^{2}\omega\phi+p\lambda_{1}\phi+2Kp \phi\leq0.$$
From (\ref{eq:3.11}), we have
\begin{equation}
\begin{aligned}
AG
\geq&\frac{2}{n}\phi G^{2}-C_{4}\phi^{1/2}G^{3/2}-\frac{C_{1}^{2}}{C_{4}R^{2}}\phi^{1/2}G^{3/2}\\
&-2KG\phi+(p-2)\lambda_{1}G\phi-p\lambda_{1}G\phi\\
&+\lambda_{2}e^{b\omega}\phi(C_{4}\phi^{1/2}G^{1/2}-bG+b(p-1)\lambda_{1}\omega)\\
&+\lambda_{1}\omega(-\frac{4}{n}\phi G(p-1)+C_{4}p\phi^{1/2}G^{1/2}).\label{eq:3.16}
\end{aligned}
\end{equation}
If
$$0>-\frac{4}{n}\phi G(p-1)+C_{4}p\phi^{1/2}G^{1/2},$$
then, from $(\ref{eq:3.16})$ we have
\begin{equation}
\begin{aligned}
AG
\geq&\frac{2}{n}\phi G^{2}-C_{4}\phi^{1/2}G^{3/2}-\frac{C_{1}^{2}}{C_{4}R^{2}}\phi^{1/2}G^{3/2}\\
&-2KG\phi-2\lambda_{1}G\phi.\label{eq:3.17}
\end{aligned}
\end{equation}
Dividing the both sides of the above inequality by $G$, we obtain
\begin{equation}
\begin{aligned}
A
\geq&\frac{2}{n}\phi G-C_{4}\phi^{1/2}G^{1/2}-\frac{C_{1}^{2}}{C_{4}R^{2}}\phi^{1/2}G^{1/2}-2K\phi-2\lambda_{1}\phi\\
\geq&\frac{1}{n}\phi G-\frac{n}{4}(C_{4}+\frac{C_{1}^{2}}{C_{4}R^{2}})^{2}-2K\phi-2\lambda_{1}\phi.\label{eq:3.18}
\end{aligned}
\end{equation}
Hence, we know that at $x_{0}$ there holds true
\begin{equation}
\begin{aligned}
\phi G\leq & \inf_{C_{4}>0}\{nA+\frac{n^{2}}{4}(C_{4}+\frac{C_{1}^{2}}{C_{4}R^{2}})^{2}+2Kn+n(p-2)\lambda_{1}+np\lambda_{1}\}\\
=& nA+\frac{n^{2}C_{1}^{2}}{R^{2}}+2Kn+n(p-2)\lambda_{1}+np\lambda_{1},\label{eq:3.19}
\end{aligned}
\end{equation}
where $$A=\frac{((n-1)(1+\sqrt{K} R)+2)C_{1}^{2}+C_2}{R^{2}}.$$
Otherwise, we have
$$0\leq-\frac{4}{n}\phi G(p-1)+C_{4}p\phi^{1/2}G^{1/2},$$
it follows that
$$0\leq-\frac{2}{n}\phi G(p-1)+\frac{n}{8(p-1)}C_{4}^{2}p^{2},$$
hence, we obtain
\begin{equation}
\begin{aligned}
\phi G\leq & \inf_{C_{4}>0}\{\frac{n}{2(p-1)}\frac{n}{8(p-1)}C_{4}^{2}p^{2} \}
=0.\label{eq:3.20}
\end{aligned}
\end{equation}
This is a trivial conclusion.

Combining $(\ref{eq:3.8}), (\ref{eq:3.14}), (\ref{eq:3.15}), (\ref{eq:3.19})$ and $(\ref{eq:3.20})$, we have
\begin{equation}
\sup_{B_{R}(O)}G
\leq G\phi
\leq \max\left\{
\begin{aligned}
&\min_{C_3\in(0,\frac{2(2-p)}{np})}\left\{\frac{((A+2K+2\lambda_{1})R^{2}C_{3}+C_{1}^{2})np}{(2(2-p)-C_{3}np)R^{2}C_{3}}\right\};\\
& nA+\frac{n^{2}C_{1}^{2}}{R^{2}}+2Kn+n(p-2)\lambda_{1}+np\lambda_{1};\\
& \frac{n}{2(p-1)}(\frac{2}{n}(p-1)^{2}L+p\lambda_{1}+2Kp )
\end{aligned}
\right\}=\tilde{C_{1}}.\label{eq:3.21}
\end{equation}
Thus, we complete the proof of Case 1.
\medskip

Case 2: $\lambda_{1}\leq0$ and $\lambda_{2}>0$. For the present situation, we need to consider the following two cases on $\omega$:
(1). $\omega\geq 0$;
(2). $\omega< 0$.
We will discuss them one by one.

(1). $\omega\geq 0$.

From ($\ref{eq:3.5}$) we have
\begin{equation}
\begin{aligned}
AG
\geq&\frac{2}{n}\phi G^{2}-C_{3}\phi G^{2}-\frac{4}{n}\phi G(p-1)\lambda_{1}\omega-\frac{C_{1}^{2}}{R^{2}}\frac{G}{C_{3}}-2KG\phi\\
&-2\lambda_{1}G\phi+\lambda_{2}e^{b\omega}\phi(C_{3}G-bG+b(p-1)\lambda_{1}\omega)\\
&+\lambda_{1}^2\phi(\frac{2}{n}(p-1)^{2}\omega^{2}+p\omega)+2Kp\lambda_{1}\omega \phi.\\
\label{eq:3.22}
\end{aligned}
\end{equation}
Noticing that
$$K\geq0, \quad\quad 0\leq\phi\leq 1, \quad\quad 1<p<2 \quad\quad\mbox{and}\quad\quad b\leq0,$$
we have the followings:
\begin{equation}
-2\lambda_{1}G\phi+\lambda_{2}e^{b\omega}\phi(C_{3}G-bG+b(p-1)\lambda_{1}\omega)\geq0,\label{eq:3.23}
\end{equation}
and
\begin{equation}
\lambda_{1}^2\phi(\frac{2}{n}(p-1)^{2}\omega^{2}+p\omega)\geq0.\label{eq:3.24}
\end{equation}
By substituting (\ref{eq:3.23}) and (\ref{eq:3.24}) into (\ref{eq:3.22}), we derive
\begin{equation}
\begin{aligned}
AG\geq & \frac{2}{n}\phi G^{2}-C_{3}\phi G^{2}-\frac{4}{n}\phi G(p-1)\lambda_{1}\omega-\frac{C_{1}^{2}}{R^{2}}\frac{G}{C_{3}}-2KG\phi
+2Kp\lambda_{1}\omega \phi\\
=&(\frac{2}{n}-C_{3})G^{2}\phi-\frac{C_{1}^{2}}{R^{2}}\frac{G}{C_{3}}-2KG\phi+\lambda_{1}\omega \phi(-\frac{4}{n} G(p-1)+2Kp).
\label{eq:3.25}
\end{aligned}
\end{equation}
If $$-\frac{4}{n} G(p-1)+2Kp<0,$$ then we have $$\lambda_{1}\omega \phi(-\frac{4}{n} G(p-1)+2Kp)\geq0.$$
Thus, from (\ref{eq:3.25}) it follows
\begin{equation}
AG\geq(\frac{2}{n}-C_{3})G^{2}\phi-\frac{C_{1}^{2}}{R^{2}}\frac{G}{C_{3}}-2KG\phi.
\label{eq:3.26}
\end{equation}
Taking $C_{3}<\frac{2}{n}$ and dividing the both sides of (\ref{eq:3.26}) by $(\frac{2}{n}-C_{3})G$, we obtain
\begin{equation}
G\phi\leq \frac{A+\frac{C_{1}^{2}}{R^{2}}\frac{1}{C_{3}}+2K}{\frac{2}{n}-C_{3}}.
\label{eq:3.27}
\end{equation}
Otherwise, we have $$-\frac{4}{n} G(p-1)+2Kp\geq0,$$
it follows that
\begin{equation}
G\phi\leq \frac{nKp}{2(p-1)}.
\label{eq:3.28}
\end{equation}
\medskip

(2). $\omega< 0$.

From ($\ref{eq:3.5}$) we have
\begin{equation}
\begin{aligned}
AG
\geq&\frac{2}{n}\phi G^{2}-\frac{4}{n}\phi G(p-1)\lambda_{1}\omega-C_{3}\phi G^{2}-\frac{C_{1}^{2}}{R^{2}}\frac{G}{C_{3}}-2KG\phi\\
&-2\lambda_{1}G\phi+\lambda_{1}^2\phi p\omega-b\lambda_{2}e^{b\omega}\phi(G-(p-1)\lambda_{1}\omega)\\
&+\lambda_{1}^2\phi\frac{2}{n}(p-1)^{2}\omega^{2}+2Kp\lambda_{1}\omega \phi+\lambda_{2}e^{b\omega}\phi C_{3}G.\\
\label{eq:3.29}
\end{aligned}
\end{equation}
Since $$G=|\nabla{\omega}|^{2}+p\lambda_{1} \omega+\lambda_{2}  e^{b\omega},$$
then we have $G\geq p\lambda_{1} \omega$. Hence,
\begin{equation}
\frac{2}{n}\phi G^{2}-\frac{4}{n}\phi G(p-1)\lambda_{1}\omega\geq(\frac{2}{n}-\frac{4(p-1)}{np})\phi G^{2}=\frac{2(2-p)}{np}\phi G^{2},
\label{eq:3.30}
\end{equation}
\begin{equation}
-2\lambda_{1}G\phi+\lambda_{1}^2\phi p\omega=-\lambda_{1}\phi(G+G- p\lambda_{1} \omega)\geq0,
\label{eq:3.31}
\end{equation}
and
\begin{equation}
-b\lambda_{2}e^{b\omega}\phi(G-(p-1)\lambda_{1}\omega)=-b\lambda_{2}e^{b\omega}\phi(G-p\lambda_{1}\omega+\lambda_{1}\omega)\geq0.
\label{eq:3.32}
\end{equation}
Thus, by substituting (\ref{eq:3.30}), (\ref{eq:3.31}) and (\ref{eq:3.32}) into (\ref{eq:3.29}), and noticing that
$$\lambda_{1}^2\phi\frac{2}{n}(p-1)^{2}\omega^{2}+2Kp\lambda_{1}\omega \phi+\lambda_{2}e^{b\omega}\phi C_{3}G\geq0,$$
we obtain
\begin{equation}
AG
\geq(\frac{2(2-p)}{np}-C_{3})\phi G^{2}-\frac{C_{1}^{2}}{R^{2}}\frac{G}{C_{3}}-2KG\phi.
\label{eq:3.33}
\end{equation}
Taking $C_{3}<\frac{2(2-p)}{np}$ and dividing the both sides of (\ref{eq:3.33}) by $(\frac{2(2-p)}{np}-C_{3}) G$, we obtain the following inequality
\begin{equation}
G\phi\leq \frac{A+\frac{C_{1}^{2}}{R^{2}}\frac{1}{C_{3}}+2K}{\frac{2(2-p)}{np}-C_{3}}.
\label{eq:3.34}
\end{equation}
Noticing that if $C_{3}\rightarrow 0^{+}$ or $C_{3}\rightarrow \left(\frac{2(2-p)}{np}\right)^{-}$, the right hand side of the above inequality tends to $+\infty$. The right hand side is continuous function of variable $C_{3}$, thus it can take its minimum in the interval $(0,\frac{2(2-p)}{np})$.

Since $0<\frac{2(2-p)}{np}<\frac{2}{n}$, combining (\ref{eq:3.27}), (\ref{eq:3.28}) and (\ref{eq:3.34}) we get
\begin{equation}
\sup_{B_{R}(O)}G
\leq G\phi
\leq \max\left\{
\begin{aligned}
&\frac{nKp}{2(p-1)};\\
& \min_{C_3\in(0,\frac{2(2-p)}{np})}\left\{\frac{((A+2K)R^{2}C_{3}+C_{1}^{2})np}{(2(2-p)-C_{3}np)R^{2}C_{3}}\right\}
\end{aligned}
\right\}=\tilde{C_{2}},\label{eq:3.35}
\end{equation}
where $$A=\frac{((n-1)(1+\sqrt{K} R)+2)C_{1}^{2}+C_2}{R^{2}}.$$
Thus, we complete the proof of Case 2 and the proof of {\bf Theorem 1.1}.
\end{proof}

\section{The proof of Theorem 1.2}
In this section, we denote $(M, g)$ an $n$-dimensional complete Riemannian manifold with $Ric(g)\geq-Kg$ in the geodesic ball $B_{2R}(O)$, where $O$ is a fixed point on $M$ and $K=K(2R)$ is a nonnegative constant depending on $R$. Let
\begin{equation}
G_{1}=|\nabla (\ln u)|^{2}+\lambda h(\ln u),\label{eq:4.1}
\end{equation}
Taking $x_{1}\in B_{2R}(O)$ such that
$$\phi G_{1}(x_{1})=\sup_{B_{2R}(O)}(\phi G_{1})\geq 0$$
and replacing $G$ by $G_{1}$ in ($\ref{eq:2.9}$), we can get

\begin{equation}
\begin{aligned}
AG_{1}
\geq &\phi\Delta G_{1}\\
\geq &(-\frac{4}{n}(\lambda -1)G_{1}+\frac{2}{n}(\lambda-1)^{2}h)h\phi\\
&+((\lambda -2)G_{1}-(\lambda-1)\lambda h)h'\phi\\
&+(\lambda G_{1}-\lambda^{2}h)h''\phi\\
&+\frac{2}{n}G_{1}^{2}\phi-2KG_{1}\phi+2K\lambda h\phi-2G_{1}|\nabla\phi|(G_{1}-\lambda h)^{1/2}\\
=&((-\frac{4}{n}(\lambda -1)h+(\lambda -2)h'+\lambda h'')G_{1}+\frac{2}{n}(\lambda-1)^{2}h^{2}-(\lambda-1)\lambda hh'-\lambda^{2}h h'')\phi\\
&+2K\lambda h\phi+\frac{2}{n}G_{1}^{2}\phi-2KG_{1}\phi\\
&-2G_{1}|\nabla\phi|(G_{1}-\lambda h)^{1/2}.\label{eq:4.2}
\end{aligned}
\end{equation}

Now, we are in the position to give the proof of Theorem 1.2.

\begin{proof}
Using Young's inequality, we know that there holds
\begin{equation}
2G_{1}|\nabla\phi|(G_{1}-\lambda h)^{1/2}\leq C_{5}\phi G_{1}(G_{1}-\lambda h)+\frac{|\nabla\phi|^{2}}{\phi}\frac{G_{1}}{C_{5}},\label{eq:4.3}
\end{equation}
where $C_{5}$ is a positive constant to be determined later.
Then, from ($\ref{eq:4.2}$) we have
\begin{equation}
\begin{aligned}
AG_{1}
\geq &((-\frac{4}{n}(\lambda -1)h+(\lambda -2)h'+\lambda h'')G_{1}+\frac{2}{n}(\lambda-1)^{2}h^{2}-(\lambda-1)\lambda hh'-\lambda^{2}h h'')\phi\\
&+2K\lambda h\phi+\frac{2}{n}G_{1}^{2}\phi-2KG_{1}\phi-C_{5}\phi G_{1}(G_{1}-\lambda h)-\frac{|\nabla\phi|^{2}}{\phi}\frac{G_{1}}{C_{5}}.\\
%\geq &(I+II+III)\phi+2K\lambda h\phi\\
%&+\frac{2}{n}G_{1}^{2}\phi-2KG_{1}\phi-C_{5}\phi G_{1}(G_{1}-\lambda h)-\frac{C_{1}^{2}}{R^{2}}\frac{G}{C_{5}}\\
%\geq &(I+II+III)\phi+2K\lambda h\phi\\
%&+\frac{2}{n}G_{1}^{2}\phi-2KG_{1}\phi-C_{5}\phi G_{1}^{2}-\frac{C_{1}^{2}}{R^{2}}\frac{G_{1}}{C_{5}}.\\
\label{eq:4.4}
\end{aligned}
\end{equation}
From (\ref{eq:1.9}),
$$-\frac{4}{n}(\lambda -1)h+(\lambda -2)h'+\lambda h''\geq0,\quad\quad\lambda h\geq0.$$
Noticing $G_{1}\geq\lambda h$, we have
\begin{equation}
\begin{aligned}
AG_{1}
\geq &((-\frac{4}{n}(\lambda -1)h+(\lambda -2)h'+\lambda h'')\lambda h+\frac{2}{n}(\lambda-1)^{2}h^{2}-(\lambda-1)\lambda hh'-\lambda^{2}h h'')\phi\\
&+2K\lambda h\phi+\frac{2}{n}G_{1}^{2}\phi-2KG_{1}\phi-C_{5}\phi G_{1}^{2}-\frac{|\nabla\phi|^{2}}{\phi}\frac{G_{1}}{C_{5}}\\
=&h(2K\lambda-\frac{2}{n}(\lambda^{2}-1)h-\lambda h')\phi\\
&+\frac{2}{n}G_{1}^{2}\phi-2KG_{1}\phi-C_{5}\phi G_{1}^{2}-\frac{|\nabla\phi|^{2}}{\phi}\frac{G_{1}}{C_{5}}
\label{eq:4.5}
\end{aligned}
\end{equation}

From (\ref{eq:1.9}),
$$h(2K\lambda-\frac{2}{n}(\lambda^{2}-1)h-\lambda h')\geq 0.$$
Noticing $0\leq \phi \leq 1$ and $K\geq0$, (\ref{eq:4.5}) turns into the follwoing
\begin{equation}
\begin{aligned}
AG_{1}=&\frac{((n-1)(1+\sqrt{K} R)+2)C_{1}^{2}+C_2}{R^{2}}G_{1}\\
\geq &(\frac{2}{n}-C_{5})G_{1}^{2}\phi-2KG_{1}\phi-\frac{C_{1}^{2}}{R^{2}}\frac{G_{1}}{C_{5}}.\\
\label{eq:4.6}
\end{aligned}
\end{equation}
When $\frac{2}{n}-C_{5}>0$, multiplying the both side of the ($\ref{eq:4.6}$) by $1/(\frac{2}{n}-C_{5})G_{1}$, we obtain
\begin{equation}
\frac{((n-1)(1+\sqrt{K} R)+2)C_{1}^{2}+C_2}{(\frac{2}{n}-C_{5})R^{2}}
\geq G_{1}\phi-\frac{2K\phi}{(\frac{2}{n}-C_{3})}-\frac{C_{1}^{2}\frac{1}{C_{5}}}{(\frac{2}{n}-C_{5})R^{2}}.\\
\label{eq:4.7}
\end{equation}
Thus, at $x_{1}$, there holds true
\begin{equation}
\begin{aligned}
\sup_{B_{R}(O)}G_{1}
\leq&G_{1}\phi\\
\leq&\frac{((n-1)(1+\sqrt{K} R)+2)C_{1}^{2}+C_2}{(\frac{2}{n}-C_{5})R^{2}}+
\frac{2K\phi}{(\frac{2}{n}-C_{3})}+\frac{C_{1}^{2}\frac{1}{C_{5}}}{(\frac{2}{n}-C_{5})R^{2}}\\
\leq&\frac{((n-1)(1+\sqrt{K} R)+2+\frac{1}{C_{5}})C_{1}^{2}+C_2+2KR^{2}}{(\frac{2}{n}-C_{5})R^{2}},\label{eq:4.8}
\end{aligned}
\end{equation}
for all $C_{5}\in(0,2/n)$. Noticing that if $C_{5}\rightarrow 0^{+}$ or $C_{5}\rightarrow (2/n)^{-}$, the right hand side of the inequality tends to $+\infty$. The right hand side is a continuous function of $C_{5}$, thus it can take its minimum in $(0,2/n)$. Then we complete the proof of {\bf Theorem 1.2}.
\end{proof}

Here we also give a brief proof of the corollaries:
\begin{proof}
For any two points $x,\, y\in{B_{R/2}(O)}$, there holds true
\begin{equation}
\ln u(x)-\ln u(y)\leq\int_{\gamma}{\frac{|\nabla u|}{u}},
\end{equation}
where $\gamma$ is a curve connecting $x$ and $y$ in $M$. Noticing $\frac{|\nabla u|}{u}\leq\sqrt{C(n, K, R, h)}$, it follows
$$\sup_{B_{R/2}(O)}u\leq e^{\int_{\gamma}{\sqrt{C(n, K, R, h)}}} \inf_{B_{R/2}(O)}u\leq e^{R\sqrt{C(n, K, R, h)}} \inf_{B_{R/2}(O)}u.$$
This is {\bf Corollary \ref{cor:1.1}}.

When $K=0$, letting $R\rightarrow +\infty$, we have
\begin{equation}
\frac{|\nabla u|}{u}\leq \sqrt{C(n, K, R, h)}\rightarrow 0.\label{eq:4.9}
\end{equation}
Then any positive solution $u$ of $(\ref{eq:1.4})$ must be constant for $|\nabla u|\equiv0$. Moreover, if $\lambda h > 0$, equation
$$\Delta u+uh(\ln u)=0$$ admits no positive solutions. This is {\bf Corollary \ref{cor:1.2}}.

By a direct calculation we can see easily that, if $h(\ln u)\geq 0$, $h'(\ln u)\leq 0$, $h''(\ln u)\geq 0$, and $0 \leq \lambda \leq 1$, we have
\begin{equation*}
\left\{
\begin{aligned}
& -\frac{4}{n}(\lambda-1)h+(\lambda-2)h'+\lambda h''\geq 0;\\
& h(2K\lambda-\frac{2}{n}(\lambda^{2}-1)h-\lambda h')\geq 0;\\
& \lambda h\geq 0.
\end{aligned}
\right.
\end{equation*}
Then (\ref{eq:1.9}) holds true, therefore, the conclusions of Theorem $\ref{them:1.2}$ hold true. This is {\bf Corollary \ref{cor:1.3}}.

For the case $\lambda=1$, (\ref{eq:1.9}) turns into
\begin{equation*}
\left\{
\begin{aligned}
& -h'+h''\geq 0;\\
& h(2K-h')\geq 0;\\
& h\geq 0.
\end{aligned}
\right.
\end{equation*}
Thus $h'(\ln u)\leq \min\{h''(\ln u),2K\}$ and $h(\ln u)\geq 0$ on $B_{R}(O)$.  This is {\bf Corollary \ref{cor:1.4}}.

On the other hand, for the case $\lambda=0$, (\ref{eq:1.9}) turns into
\begin{equation*}
\left\{
\begin{aligned}
& \frac{4}{n}h-2h'\geq 0;\\
& \frac{2}{n}h^{2}\geq 0.
\end{aligned}
\right.
\end{equation*}
Thus $\frac{2}{n}h(\ln u)-h'(\ln u)\geq 0$ on $B_{R}(O)$. This is {\bf Corollary \ref{cor:1.5}}. And we complete the proof of the Corollaries.
\end{proof}
\noindent {\it\textbf{Acknowledgements}}: The authors are supported partially by NSFC grant (No.11731001). The author Y. Wang is supported partially by NSFC grant (No.11971400) and Guangdong Basic and Applied Basic Research Foundation Grant (No. 2020A1515011019).

\end{document}